\documentclass[12pt,letterpaper]{amsart}

\newcounter{myenum}

\newcommand{\la}{\langle}
\newcommand{\ra}{\rangle}

\newcommand{\I}{\operatorname{Im}}
\newcommand{\ds}{\displaystyle}

\newcommand{\nls}{\operatorname{NLS}}
\newcommand{\egb}{\exists \hspace{-.1cm}\operatorname{GB}}

\usepackage{amssymb}
\usepackage{amsthm}
\usepackage{amsxtra}
\usepackage{graphicx}

\newtheorem{theorem}{Theorem}
\newtheorem{definition}[theorem]{Definition}
\newtheorem{proposition}[theorem]{Proposition}
\newtheorem{lemma}[theorem]{Lemma}
\newtheorem{corollary}[theorem]{Corollary}

\theoremstyle{remark}
\newtheorem{remark}[theorem]{Remark}

\newcommand{\cR}{\mathbb{R}}

\numberwithin{equation}{section}
\numberwithin{theorem}{section}

\ifx\pdfoutput\undefined
  \DeclareGraphicsExtensions{.pstex, .eps}
\else
  \ifx\pdfoutput\relax
    \DeclareGraphicsExtensions{.pstex, .eps}
  \else
    \ifnum\pdfoutput>0
      \DeclareGraphicsExtensions{.pdf}
    \else
      \DeclareGraphicsExtensions{.pstex, .eps}
    \fi
  \fi
\fi

\title[Divergence of 3d NLS solutions]
{Divergence of infinite-variance nonradial solutions to the 3d  NLS
equation}

\author{Justin Holmer}
\address{Brown University}
\author{Svetlana Roudenko}
\address{Arizona State University}

\begin{document}

\begin{abstract}
We consider solutions $u(t)$ to the 3d focusing NLS equation $i\partial_t u +
\Delta u + |u|^2u=0$ such that $\|xu(t)\|_{L^2} = \infty$ and $u(t)$
is nonradial.  Denoting by $M[u]$ and $E[u]$, the mass and energy,
respectively, of a solution $u$, and by $Q(x)$ the ground state
solution to $-Q+\Delta Q+|Q|^2Q=0$, we prove the following: if
$M[u]E[u]<M[Q]E[Q]$ and $\|u_0\|_{L^2}\|\nabla
u_0\|_{L^2}>\|Q\|_{L^2}\|\nabla Q\|_{L^2}$, then either $u(t)$
blows-up in finite positive time or $u(t)$ exists globally for all
positive time and there exists a sequence of times $t_n\to +\infty$
such that $\|\nabla u(t_n)\|_{L^2} \to \infty$. Similar statements
hold for negative time.
\end{abstract}

\maketitle

\section{Introduction}

The 3d focusing nonlinear Schr\"odinger equation (NLS) is
\begin{equation}
 \label{E:NLS}
i\partial_t u + \Delta u + |u|^2u=0 \, ,
\end{equation}
where $u=u(x,t)\in \mathbb{C}$ and $(x,t)\in \mathbb{R}^{3+1}$. We
shall denote the initial data $u_0(x)=u(x,0)$.  The standard local
theory in $H^1$ is based upon the Strichartz estimates (see Cazenave
\cite{C}, Tao \cite{T}), and asserts the existence of a maximal
forward time $T^* \gtrsim \|\nabla u_0\|_{L^2}^{-4}$ such that $u(t)\in
C([0,T^*); H_x^1)$.  If $T^*<\infty$, then it follows from the local
theory that $\|\nabla u(t)\|_{L^2} \to +\infty$ as $t\nearrow T^*$,
and we say that $u(t)$ \emph{blows-up} in finite forward time.  If,
on the other hand, $T^*=+\infty$, then we say that $u(t)$ exists
globally in (forward) time.  In this case, the local theory gives us
no information about the behavior of $\|\nabla u(t)\|_{L^2}$ as
$t\to +\infty$.  Analogous statements hold backwards in time.  In
fact, if $u(t)$ solves NLS, then $\bar u(-t)$ solves NLS, and thus,
it suffices to study the forward-in-time case\footnote{This is not
to say that a given solution $u(t)$ must have the same
forward-in-time and backward-in-time behavior; however, if $u_0$ is
real-valued, then $u(t)=\bar u(-t)$.}.

Solutions to \eqref{E:NLS} in $H^1$ satisfy mass, energy, and
momentum conservation, given respectively by
$$
M[u] = \|u\|_{L^2}^2, \quad E[u]= \frac12\|\nabla u\|_{L^2}^2 -
\frac14\|u\|_{L^4}^4 , \quad P[u] = \I \int \bar u \; \nabla u\,.
$$

There exists a ground state (minimal $L^2$ norm) solution $Q=Q(x)$ to the (stationary) nonlinear elliptic equation
$$-Q + \Delta Q + |Q|^2Q=0 \,,$$
which is unique modulo translation and gauge symmetry.  
This $Q$ is radial, smooth, positive, and behaves as $Q(x) \sim e^{-|x|}$ for $|x|\to +\infty$.  It gives rise to a solution $u(x,t) = e^{it}Q(x)$ to \eqref{E:NLS} called the ground state soliton.

In Holmer-Roudenko \cite{HR}, we proved that if $u_0\in H^1$,
$\|u_0\|_{L^2}\|\nabla u_0\|_{L^2} > \|Q\|_{L^2}\|\nabla Q\|_{L^2}$
and $M[u]E[u]<M[Q]E[Q]$, then $u(t)$ blows-up in finite forward (and
finite backward) time, provided that either (1)
$\|xu_0\|_{L^2}<\infty$, that is, the initial data (and hence, the
whole flow $u(t)$) has finite variance, or (2) $u_0$ (and hence, the
whole flow $u(t)$) is radial. Moreover, it is sharp in the sense
that $u(t)=e^{it}Q(x)$ solves NLS and does not blow-up in finite
time. Via the Galilean transform and momentum conservation, if
$P[u]\neq 0$, this can be refined to the following:  if
$M[u]E[u]-\frac12P[u]^2<M[Q]E[Q]$ and $\|u_0\|_{L^2}\|\nabla
u_0\|_{L^2} > \|Q\|_{L^2}\|\nabla Q\|_{L^2}$, then the above
conclusions hold (see Appendix \ref{S:Galilean} for clarification).
These results are essentially classical.  The finite variance case
follows from the virial identity \cite{VPT}, \cite{G}:
$$
\partial_t^2 \|xu(t)\|_{L^2}^2 = 24E[u]-4\|\nabla u\|_{L^2}^2
$$
and the sharp Gagliardo-Nirenberg inequality \cite{W83}. The radial
case follows from a localized virial identity and a radial
Gagliardo-Nirenberg inequality \cite{S}. The radial case is an
extension of a result of Ogawa-Tsutsumi \cite{OT91}, who proved the
case $E[u]<0$. Martel in \cite{M97} showed that in the case of $E<0$
either finite variance or radiality assumptions can be relaxed to
nonisotropic ones, namely, if (1) $\| \,|y| \, u_0 \|_{L^2_x} <
\infty$ where $y = (x_1,x_2)$, or (2) $u_0(x_1,x_2,x_3) =
u_0(|y|,x_3)$.

In this paper, we drop the additional hypothesis of finite variance
and radiality and obtain the following conclusion:

\begin{theorem}
 \label{T:main}
Suppose $u_0\in H^1$, $M[u]E[u]<M[Q]E[Q]$ and $\|u_0\|_{L^2}\|\nabla
u_0\|_{L^2}> \|Q\|_{L^2}\|\nabla Q\|_{L^2}$.  Then either $u(t)$
blows-up in finite forward time or $u(t)$ is forward global and
there exists a sequence $t_n\to +\infty$ such that $\|\nabla
u(t_n)\|_{L^2} \to +\infty$.  A similar statement holds for negative
time.
\end{theorem}

It is still possible, as far as we know, that a given solution
satisfying the hypothesis might, say, blow-up in finite negative time but
be global in forward time with the existence of a sequence $t_n\to
+\infty$ such that $\|\nabla u_n(t)\|_{L^2}\to +\infty$.  In other
words, a given solution might have different behavior in forward and
backward times.

The above remarks regarding the refinement for $P[u]\neq 0$, by
applying a Galilean transformation to convert to a solution with
$P[u]=0$, apply in the context of Theorem \ref{T:main} as well.  In
fact, we will always assume $P[u]=0$ in this paper (see Appendix
\ref{S:Galilean} for the standard details).

A result similar to Theorem \ref{T:main} was obtained by
Glangetas-Merle \cite{GM} for the case of $E[u]<0$ (see also Nawa
\cite{N}). However, our proof is different in structure and uses a
different form of concentration compactness machinery.  Our proof is
more akin to the proof of the scattering result we have in
\cite{HR}, \cite{DHR}, appealing to (suitable adaptations of) the
profile decomposition results of Keraani \cite{K}, nonlinear
perturbation theory based upon the Strichartz estimates, and
rigidity theorems based upon the localized virial identity.  Our
scattering result was in turn modeled on a similar result by
Kenig-Merle \cite{KM} for the energy-critical NLS equation.  In his
various lectures, Kenig refers to this scheme as the ``concentration
compactness--rigidity method'' and discusses a ``road map'' for
applying it to various problems. We believe that this method applied
to prove Theorem \ref{T:main} has more potential for generalization.
In particular, it could perhaps provide an affirmative answer to:

\medskip

\noindent\textbf{Weak conjecture}.  Under the hypothesis of Theorem
\ref{T:main}, either $u(t)$ blows-up in finite forward time or
$\|\nabla u(t)\|_{L^2} \to \infty$ as $t\to +\infty$.

\medskip

\noindent\textbf{Strong conjecture}. Under the hypothesis of Theorem
\ref{T:main}, $u(t)$ blows-up in finite forward time.

\medskip

Why are we interested in removing the finite-variance hypothesis
from our earlier result?  The assumption $\|xu_0\|_{L^2}< \infty$
might be considered unnatural on the grounds that blow-up is a
local-in-space phenomenon and should not be dictated, in such a
strong sense, by the size of the initial data at spatial infinity.
In the case $\|xu_0\|_{L^2}<\infty$ addressed in \cite{HR}, the
proof given via the virial identity actually provides, once the
solution is scaled so that $M[u]=M[Q]$, an upper bound $T_b$ on the
blow-up time $T^*$, where $T_b$ is given as:
$$
T_b = r'(0) + \sqrt{r'(0)^2+2r(0)} \,,
$$
and
$$
r(0) = c_1 \|x u_0\|_{L^2}^2\,, \quad r'(0) = 4 c_1 \I \int (x\cdot
\nabla u_0) \; \bar u_0 \,.
$$
Here, $c_1$ is a constant depending on $E[u]$ that diverges as
$E[u]\nearrow E[Q]$.  We carry out this classical argument in Prop.
\ref{P:blowup-time}. This upper bound is actually an estimate for
the time at which $\|xu(t)\|_{L^2}=0$ if $u(t)$ were to continue to
exist up to that time. However, numerics show that even if blow-up
occurs at the origin, the variance $\|xu(t)\|_{L^2}$ actually does
not go to zero at the blow-up time due to radiated mass ejected from
the blow-up core, and thus, blow-up occurs before the time predicted
by this method.  This suggests that the full variance
$\|xu(t)\|_{L^2}$ is not the correct quantity on which to base a
blow-up theory.    An analysis of the radial case using the radial
Gagliardo-Nirenberg inequality (carried out in Prop.
\ref{P:radial-blowup-time}) reveals that  there is an upper bound
expressible entirely in terms of a spatially truncated version of
$r'(0)$ as well as the proximity of $E[u]$ to $E[Q]$.  Thus, the
size of the initial variance does not appear at all, and $r'(0)$ can
be thought of as measuring the degree and sign of quadratic
phase.\footnote{The relevance of quadratic phase seems  very
important from our numerics, see forthcoming paper \cite{HPR09}.  We
remark that in the 2d case it is exactly quantifiable via the
pseudoconformal transformation.} Theorem \ref{T:main} might be
considered the first step in assessing the relevance of the variance
in blow-up theory of nonradial solutions, even though it is,
unfortunately, nonquantitative.\footnote{Another problem we
face in the nonradial case is that of predicting the \emph{location}
of the blow-up. Nothing says that blow-up should occur at the
origin, even if $P[u]=0$.}

Another motivation is that there exist equations with less structure
that NLS, such as the Zakharov system, for which the assumption of
finite variance is not known to be of assistance in proving that
negative energy solutions blow-up.  Merle \cite{Mer96a} proved using
a localized virial-type identity that \emph{radial} negative energy
solutions of the 3d Zakharov system behave according to the
conclusion of Theorem \ref{T:main}.  No result is known for
nonradial solutions (finite-variance or not) and it is conceivable
that the concentration compactness methods of this paper might be of
assistance in addressing this case.  Even for the 3d NLS equation
\eqref{E:NLS} itself, there are studies in the behavior of
finite-time blow-up solutions, such as the divergence of the
critical $L^3$ norm proved for radial solutions in Merle-Raphael
\cite{MR}, for which concentration compactness methods might enable
one to remove the radiality assumption.

The paper is structured as follows. \S \ref{S:prelim}--\ref{S:pd}
are devoted to preparatory material; \S
\ref{S:outline-induction}--\ref{S:concentration} are devoted the
proof of Theorem \ref{T:main}.  In \S \ref{S:prelim}, we review the
dichotomy and scattering result we obtained in \cite{HR},
\cite{DHR}. In \S \ref{S:virial} we deduce some blow-up theorems for
the virial identity and its localized versions -- in the nonradial
case, we are forced to assume an \emph{a priori} uniform-in-time
localization on the solution under consideration.  In \S
\ref{S:varQ}, we rewrite the variational characterization of the
ground state $Q$ from Lions \cite{L84} in a form that is more
compatible with the scale-invariant perspective of this paper; this
material is needed for \S \ref{S:near-boundary}.  In \S
\ref{S:near-boundary}, we carry out the base-case of the inductive
argument that follows in
\S\ref{S:outline-induction}--\ref{S:concentration}.  Under the
assumption that Theorem \ref{T:main} is false, we are able to
construct a special ``critical'' solution that remains
uniformly-in-time concentrated in $H^1$.  Such a solution would
contradict the results of \S \ref{S:virial}, and hence, cannot
exist.

\subsection{Acknowledgements}
The second author would like to thank Patrick G\'erard for
discussions leading to the questions addressed in this paper and
help in retrieving the reference \cite{GM}. J. H. is partially
supported by a Sloan fellowship and NSF Grant DMS-0901582. 
S. R. is partially supported by NSF-DMS grant \# 0808081.

\section{Ground state and dichotomy}
\label{S:prelim}

We begin by recalling a few basic facts about the ground state $Q$, the minimal
mass $H^1(\cR^3)$ solution of $-Q + \Delta Q + |Q|^2Q = 0$.

Weinstein \cite{W83} proved that the sharp constant $c_{GN}$ of
Gagliardo-Nirenberg inequality
\begin{equation}
 \label{E:GN}
\|u\|_{L^4(\cR^3)}^4 \leq c_{GN} \|u\|_{L^2(\cR^3)} \|\nabla
u\|_{L^2(\cR^3)}^3
\end{equation}
is achieved by taking $u=Q$.  Using the Pohozhaev identities
$$
\|\nabla Q\|_2^2 = 3 \|Q\|_2^2, \quad \|Q\|_4^4 = 4 \|Q\|_2^2,
$$
we can express $c_{GN}$ as
\begin{equation}
\label{E:cGN}
c_{GN} = \frac{4}{3\|Q\|_2\|\nabla Q\|_{2}} \,.
\end{equation}
The Pohozhaev identities also give:
\begin{equation}
\label{E:ME}
E[Q] = \frac16\|\nabla Q\|_{L^2}^2 \,.
\end{equation}
Let
\begin{equation}
 \label{E:eta-def}
\eta(t) = \frac{\|u(t)\|_{L^2}\|\nabla
u(t)\|_{L^2}}{\|Q\|_{L^2}\|\nabla Q\|_{L^2}} \,.
\end{equation}
By \eqref{E:GN}, \eqref{E:cGN}, and \eqref{E:ME}, we have
\begin{equation}
\label{E:energy1} 3\eta(t)^2 \geq \frac{M[u] E[u]}{M[Q]E[Q]} \geq
3\eta(t)^2-2\eta(t)^3 ,
\end{equation}
see Figure \ref{F:dichotomy}.

\begin{figure}[h]
\includegraphics[width=400pt]{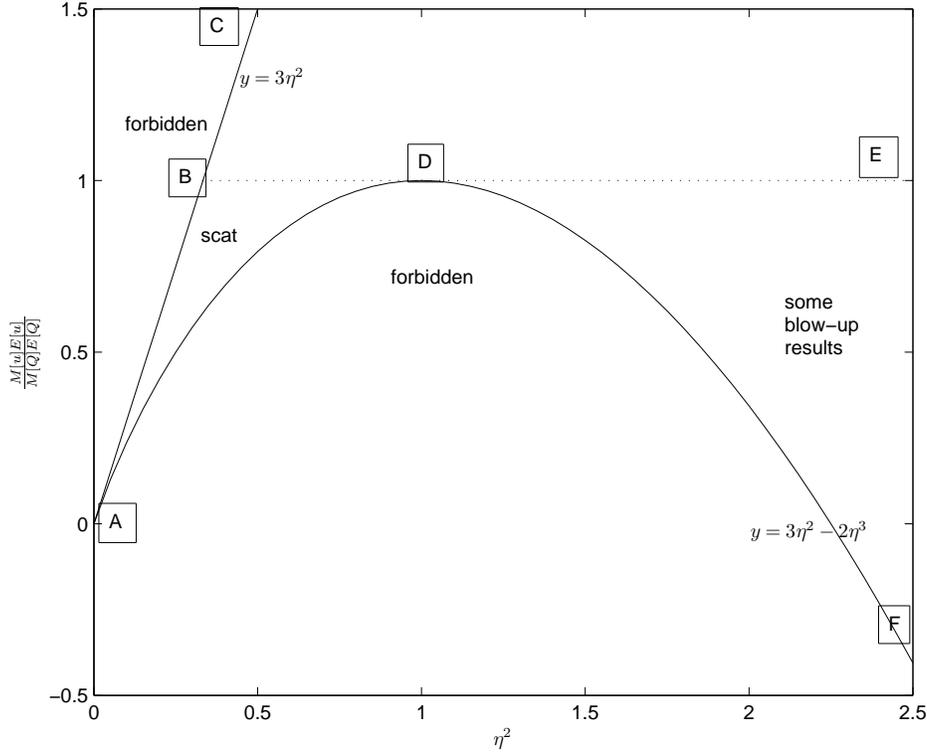}
\caption{A plot of $M[u]E[u]/M[Q]E[Q]$ versus $\eta^2$, where $\eta$
is defined by \eqref{E:eta-def}.  The area to the left of line ABC
and inside region ADF are excluded by \eqref{E:energy1}.  The region
inside ABD corresponds to case (1) of Prop. \ref{P:dichotomy} and
Theorem \ref{T:scattering} (solutions scatter).  The region EDF
corresponds to case (2) of Prop. \ref{P:dichotomy} and Theorem
\ref{T:main} (solutions either blow-up in finite time or diverge in
$H^1$ along a sequence $t_n\to \infty$), Prop. \ref{P:blowup-time}
(finite-variance solutions blow-up in finite time), and Prop.
\ref{P:radial-blowup-time} (radial solutions blow-up in
finite-time).   Behavior of solutions on the dotted line
(mass-energy threshold line) is given by \cite[Theorem 3-4]{DR}. }
 \label{F:dichotomy}
\end{figure}

Suppose that $M[u]E[u]/M[Q]E[Q]<1$. Then we have 2 cases:
\begin{itemize}
\item
If $0 \leq M[u]E[u]/M[Q]E[Q]<1$, then there exist two solutions (see
Figure \ref{F:evolution}) $0 \leq \lambda_-<1<\lambda$ to
\begin{equation}
 \label{E:energy-lambda}
\frac{M[u] E[u]}{M[Q]E[Q]} = 3\lambda^2-2\lambda^3 \,.
\end{equation}
\item
If $E[u]<0$, then there exists exactly one solution $\lambda>1$ to
\eqref{E:energy-lambda}.
\end{itemize}
By the $H^1$ local theory, there exist $-\infty\leq T_*<0<T^*\leq
\infty$ such that $T_*< t < T^*$ is the maximal time interval of
existence for $u(t)$ solving \eqref{E:NLS}. Moreover,
$$
T^*<+\infty \implies \|\nabla u(t)\|_{L^2} \geq
\frac{c}{(T^*-t)^{1/4}} \quad \text{ as }t\nearrow T^*
$$
with a similar statement holding if $T_*\neq -\infty$.

\begin{figure}
\includegraphics[width=400pt]{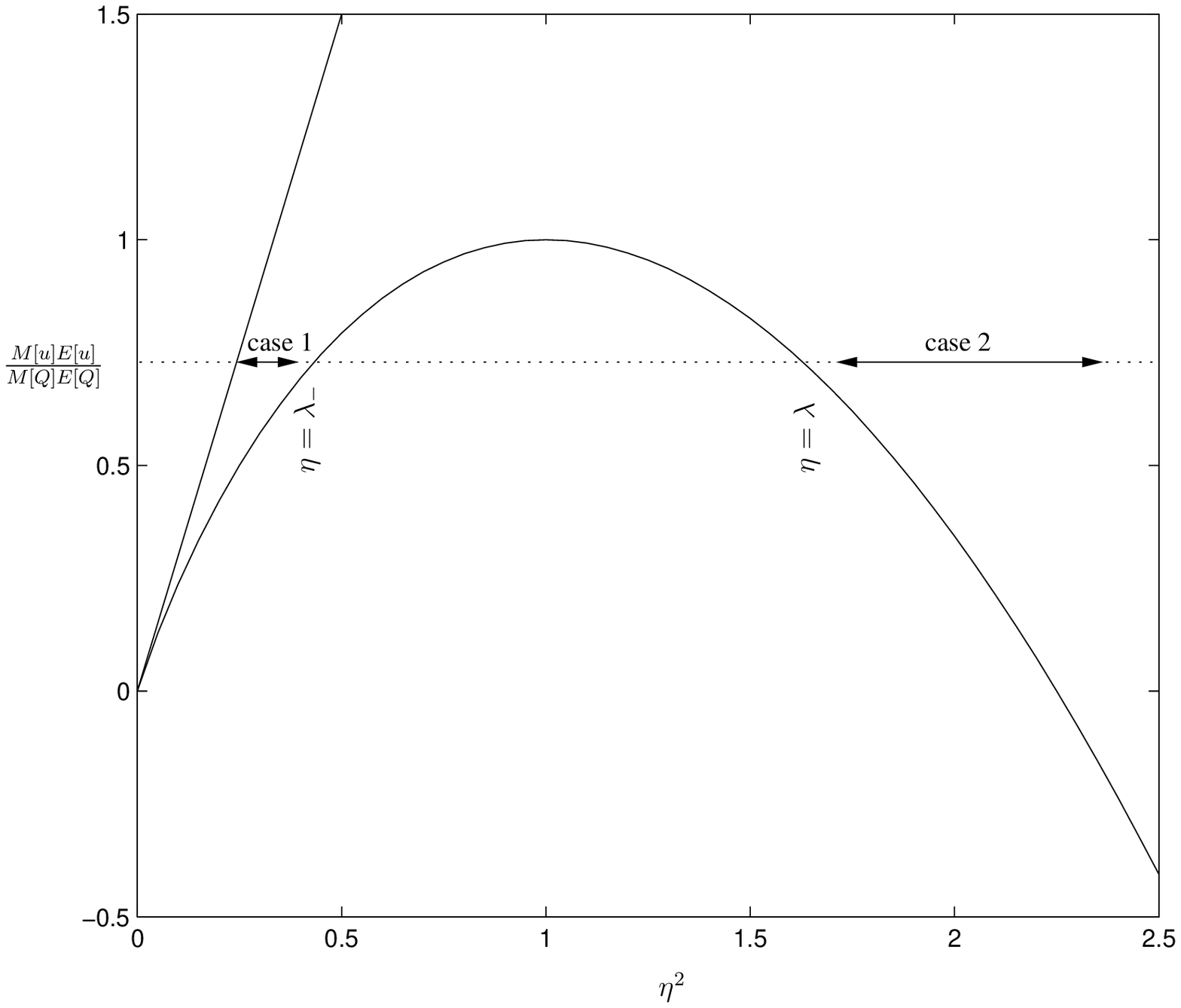}
\caption{On the plot of $M[u]E[u]/M[Q]E[Q]$ versus $\eta^2$,
indicates how a choice of $M[u]E[u]/M[Q]E[Q]$ determines via
\eqref{E:energy-lambda} (at most) two special values of $\eta$,
namely $\eta=\lambda_-$ and $\eta=\lambda$.  In the NLS flow in Case
1 and Case 2 of Prop. \ref{P:dichotomy}, $\eta(t)$ moves along the
indicated horizontal lines.  Note that Theorem \ref{T:scattering}
states that in Case 1, $\eta(t)$ approaches the left endpoint as
$t\to \pm \infty$. Theorem \ref{T:main} states that in Case 2, there
exists a sequence of times $t_n\to +\infty$ along which
$\eta(t_n)\to +\infty$.}
 \label{F:evolution}
\end{figure}

The following is a consequence of the continuity of the flow $u(t)$
(see Figures \ref{F:dichotomy}--\ref{F:evolution}).  The proof is
carried out in \cite[Theorem 4.2]{HR}.

\begin{proposition}[dichotomy]
 \label{P:dichotomy}
Let $M[u]E[u]<M[Q]E[Q]$ and $0 \leq \lambda_- < 1 <\lambda$ be
defined as above. Then exactly one of the following holds:
\begin{enumerate}
\item
The solution $u(t)$ is global (i.e., $T_*=-\infty$ and
$T^*=+\infty$) and
$$
\forall \; t\in (-\infty,+\infty) \,,
\quad \frac13 \cdot \frac{M[u]E[u]}{M[Q]E[Q]} \leq \eta(t)^2 \leq \lambda_-^2 \,.
$$
\item
$\forall \; t\in (T_*,T^*) \,, \quad \lambda \leq \eta(t)\,.$
\end{enumerate}
The first case is only possible for $0\leq M[u]E[u]/M[Q]E[Q]\leq 1$.
\end{proposition}

Naturally, one can check the initial data (the value of $\eta(0)$)
to determine whether the solution is of the first or second type in
Prop. \ref{P:dichotomy}.  Note that the second case does not assert
finite-time blow-up (this is the subject of this paper).  In the
first case, we proved in \cite{HR}, \cite{DHR} that more holds.

\begin{theorem}[scattering]
\label{T:scattering} If $0<M[u]E[u]/M[Q]E[Q]<1$ and the first case
of Prop. \ref{P:dichotomy} holds, then $u(t)$ scatters as $t\to
+\infty$ or $t\to -\infty$.  This means that there exist $\psi_\pm
\in H^1$ such that
\begin{equation}
\label{E:asymp-linear}
\lim_{t\to \pm \infty} \|u(t)- e^{-it\Delta}\psi_\pm\|_{H^1} =0 \,.
\end{equation}
Consequently, we have that
\begin{equation}
\label{E:L4asymp}
\lim_{t\to \pm \infty} \|u(t)\|_{L^4} = 0
\end{equation}
and
\begin{equation}
\label{E:H1asymp}
\lim_{t\to \pm \infty} \eta(t)^2 =  \frac13 \cdot \frac{M[u]E[u]}{M[Q]E[Q]}\,.
\end{equation}
\end{theorem}

Let us justify \eqref{E:L4asymp}--\eqref{E:H1asymp} since they are
not mentioned in \cite{HR}, \cite{DHR}.  By \eqref{E:asymp-linear},
the Gagliardo-Nirenberg inequality, and mass conservation for the
linear and nonlinear flows, we have
$$
\lim_{t\to \pm \infty} \|u(t)- e^{-it\Delta}\psi_\pm\|_{L^4} =0 \,.
$$
The statement in \eqref{E:L4asymp} then follows by the linear decay
estimate $\|e^{-it\Delta}\psi \|_{L^4} \leq t^{-3/4}
\|\psi\|_{L^{4/3}}$ and an approximation argument (to deal with the
fact that $\psi\notin L^{4/3}$).\footnote{As a result of this
approximation argument, we lose the quantitative estimate of
$t^{-3/4}$ on the rate of decay.} By \eqref{E:L4asymp}, we have
$$
\lim_{t\to \pm \infty} \|\nabla u(t)\|_{L^2}^2 = 2E[u] +
\frac12\lim_{t\to\pm \infty} \|u(t)\|_{L^4}^4 = 2E[u] \,.
$$
Multiply by $M[u]/M[Q]E[Q]$ and use the Pohozhaev identities to
obtain \eqref{E:H1asymp}.

\section{Virial identity and blow-up conditions}
\label{S:virial}

Now we turn our attention to the second case of Prop.
\ref{P:dichotomy}.  We begin by giving the classical derivation,
using the virial identity, of the upper bound on the (finite)
blow-up time under the finite variance hypothesis.

\begin{proposition}[Finite-variance blow-up time]
 \label{P:blowup-time}
Let $M[u]=M[Q]$ and $E[u]/E[Q]<1$ and suppose that the second case
of Prop. \ref{P:dichotomy} holds (take $\lambda>1$ to be as defined
in \eqref{E:energy-lambda}). Define $r(t)$ to be the scaled
variance:
$$
r(t) =  \frac{\|xu\|_{L^2}^2}{48 E[Q] \lambda^2(\lambda-1)} \,.
$$
Then blow-up occurs in forward time before $t_b$ (i.e., $T^* \leq
t_b$), where
$$
t_b = r'(0) + \sqrt{r'(0)^2 + 2r(0)}.
$$
\end{proposition}

Note that
$$
r(0) = \frac{1}{48 E[Q] \lambda^2(\lambda-1)} \|xu_0\|_{L^2}^2
$$
and
$$
r'(0) = \frac{1}{12 E[Q] \lambda^2(\lambda-1)} \I \int (x\cdot
\nabla u_0) \; \bar u_0.
$$
As we remarked in the introduction, we feel that the dependence of
$t_b$ on $r'(0)$ (or ideally a spatially truncated version of it) is
quite natural, but the dependence on $r(0)$ seems unsubstantiated,
placing a very strong weight on the size of the solution at spatial
infinity.

\begin{proof}
The virial identity gives
$$
r''(t) = \frac{1}{48E[Q]\lambda^2(\lambda-1)}(24E[u]-4\|\nabla
u\|_{L^2}^2).
$$
By the Pohozhaev identities,
$$
r''(t) = \frac{1}{2\lambda^2(\lambda-1)}\left( \frac{E[u]}{E[Q]} -
\frac{\|\nabla u\|_{L^2}^2}{\|\nabla Q\|_{L^2}^2}\right) \,.
$$
By definition of $\lambda$ and $\eta$,
$$
r''(t) = \frac{1}{2\lambda^2(\lambda-1)}(3\lambda^2-2\lambda^3-\eta(t)^2) \,.
$$
Since $\eta(t) \geq \lambda$ (and $\lambda > 1$), we have
$$
r''(t) \leq -1 \,.
$$
Integrating in time twice gives
$$
r(t) \leq -\frac12 t^2 + r'(0)t + r(0) \,.
$$
The positive root of the polynomial on the right-hand side is $t_b$
as given in the proposition statement.
\end{proof}

We next review the local virial identity.
Let $\varphi \in C_c^\infty(\mathbb{R}^3)$ be radial such that
$$
\varphi(x) =
\begin{cases} |x|^2 &\text{for }|x|\leq 1 \\ 0 &\text{for }|x|\geq 2
\end{cases} .
$$
For $R>0$ define
\begin{equation}
 \label{E:locvar}
z_R(t) = \int R^2 \varphi\left(\frac{x}{R}\right)|u(x,t)|^2 \, dx.
\end{equation}
Then direct calculation gives the local virial identity:
\begin{equation}
 \label{E:locvar''}
z''_R(t) = 4\int \partial_j\partial_k \psi\Big( \frac{x}{R} \Big) \;
\partial_ju \, \partial_k\bar u - \int \Delta\psi \Big( \frac{x}{R}
\Big) \, |u|^4 - \frac{1}{R^2}\int \Delta^2\psi\Big( \frac{x}{R}
\Big) \, |u|^2 \,.
\end{equation}
Note that
$$
z_R''(t) = (24E[u]-4\|\nabla u(t)\|_{L^2}^2) + A_R(u(t)),
$$
where, for suitable $\varphi$ \footnote{Note that in the upper bound
we do not need the term $\|\nabla u\|_{L^2(|x|\geq R)}^2$.  This
term was needed in the \emph{lower} bound that was applied in the
proof of the scattering theorem \cite{DHR}.}
\begin{equation}
 \label{E:ARestimate}
A_R(u(t)) \lesssim  \frac{1}{R^2}
\|u\|_{L^2(|x|\geq R)}^2 + \|u\|_{L^4(|x|\geq R)}^4.
\end{equation}

Using the local virial identity, we can prove a version of Prop.
\ref{P:blowup-time}, valid without the assumption of finite variance
but assuming that the solution is suitably localized in $H^1$ for
all times.  Define
$$
\eta_{\,\geq R}(t) = \frac{\|u\|_{L^2(|x|\geq R)}\|\nabla u
\|_{L^2(|x|\geq R)}}{\|Q\|_{L^2} \|\nabla Q\|_{L^2}} \,.
$$

\begin{proposition}[Blow-up time for \textit{a priori} localized solutions]
 \label{P:local-blowup-time}
Let $M[u]=M[Q]$ and $E[u]<E[Q]$ and suppose that the second case of
Prop. \ref{P:dichotomy} holds (take $\lambda>1$ to be as defined in
\eqref{E:energy-lambda}). Select $\gamma$ such that
$0<\gamma<\min(\lambda-1,\gamma_0)$, where $\gamma_0$ is an absolute
constant. Suppose that there is a radius $R\gtrsim \gamma^{-1/2}$
such that for all $t$, there holds $\eta_{\,\geq R}(t) \lesssim
\gamma$. Define $\tilde r(t)$ to be the scaled local variance:
$$
\tilde r(t) =  \frac{z_R(t)}{48 E[Q] \lambda^2(\lambda-1-\gamma)} \,.
$$
Then blow-up occurs in forward time before $t_b$ (i.e., $T^* \leq
t_b$), where
$$
t_b = \tilde r'(0) + \sqrt{\tilde r'(0)^2 + 2\tilde r(0)} \,.
$$
\end{proposition}
One could, in fact, define $\eta_{\,\geq R}(t) =
\|u(t)\|_{L^3(|x|\geq R)}$ and obtain the same statement with a
similar proof but a different Gagliardo-Nirenberg inequality.
\begin{proof}
By the local virial identity and the same steps used in the proof of Prop. \ref{P:blowup-time},
$$
\tilde r''(t) =
\frac{1}{2\lambda^2(\lambda-1-\gamma)}\left(3\lambda^2-2\lambda^3-\eta(t)^2\right)
+ \frac{A_R(u(t))}{48 E[Q] \lambda^2(\lambda-1-\gamma)}.
$$
By the estimates (the first one is the exterior version of
Gagliardo-Nirenberg)
\begin{equation}
 \label{E:GNoutside}
\|u\|_{L^4(|x|\geq R)}^4 \lesssim \|u\|_{L^2(|x|\geq R)} \|\nabla
u\|_{L^2(|x|\geq R)}^3  \lesssim \eta_{\,\geq R}(t) \, \eta(t)^2
\lesssim \gamma \, \eta(t)^2\,
\end{equation}
and
\begin{equation}
 \label{E:truncmass}
\frac{1}{R^2} \, \|u\|_{L^2(R\leq |x|\leq 2R)}^2 \leq \frac{1}{R^2}
M[Q] \lesssim \gamma \lesssim \gamma \,\eta(t)^2\,,
\end{equation}
applied to control the $A_R$ term, and using that $\eta(t)\geq
\lambda$, we obtain
$$
\tilde r''(t) \leq -1  \,.
$$
The remainder of the argument is the same as in the proof of Prop. \ref{P:blowup-time}.
\end{proof}

For comparison purposes, we review the quantified proof of
finite-time blow-up for \emph{radial} solutions presented in
\cite{HR}.

\begin{proposition}[Radial blow-up time]
 \label{P:radial-blowup-time}
Let $M[u]=M[Q]$ and $E[u]<E[Q]$ and suppose that the second case of
Prop. \ref{P:dichotomy} holds (take $\lambda>1$ to be as defined in
\eqref{E:energy-lambda}). Suppose that $u$ is radial.  Let
$$
R = c_2 \max\left( 1, \frac{1}{\lambda^{1/2}(\lambda-1)^{1/2}} \right)\,,
$$
where $c_2$ is an appropriately large, but absolute, constant.
Define $\tilde r(t)$ to be the scaled local variance:
$$
\tilde r(t) =  \frac{z_R(t)}{48 E[Q] \lambda^2(\lambda-1)} \,.
$$
Then blow-up occurs in forward time before $t_b$ (i.e., $T^* \leq
t_b$), where
$$
t_b = \tilde r'(0) + \sqrt{\tilde r'(0)^2 + 2\,\tilde r(0)} \,.
$$
We have that
$$
t_b \lesssim \, c_\lambda \left(1+r'(0)^2\right)^{1/2} \,.
$$
where $c_\lambda \nearrow \infty$ as $\lambda \searrow 1$ (i.e., as
$E[u]\nearrow E[Q]$).
\end{proposition}
\begin{proof}
We modify the proof of Prop. \ref{P:local-blowup-time} only in
\eqref{E:GNoutside} and \eqref{E:truncmass} by using the radial
Gagliardo-Nirenberg inequality \cite{S} instead of
\eqref{E:GNoutside}
$$
\|u\|_{L^4(|x|\geq R)}^4 \lesssim \frac{1}{R^2} \|u\|_{L^2(|x|\geq
R)}^3 \|\nabla u\|_{L^2(|x|\geq R)}  \lesssim \frac{\eta(t)}{R^2}\,
$$
and also
$$
\frac{1}{R^2} \|u\|_{L^2(R\leq |x|\leq 2R)}^2 \lesssim \frac{1}{R^2}
\lesssim \frac{\eta(t)}{R^2}.
$$
Then we have, for some absolute constant $c_1$,
$$
\tilde r''(t) \leq \frac{1}{2\lambda^2(\lambda-1)}\left(
3\lambda^2-2\lambda^3-\eta(t)\left(\eta(t) -
\frac{c_1}{R^2}\right)\right) \,.
$$
We require that $R$ is large enough so that $c_1/R^2 \leq 1$. Since
$\eta(\eta- c_1/R^2)$ increases as $\eta \geq 1$ increases, and
$\eta\geq \lambda$, we have
$$
\eta(t)\left(\eta(t)-\frac{c_1}{R^2}\right) \geq
\lambda\left(\lambda - \frac{c_1}{R^2}\right) \,.
$$
This gives
$$
\tilde r''(t) \leq - \frac{1}{(\lambda-1)}\left(\lambda -1 -
c_1\frac{1}{2 \lambda \,R^2}\right) = -1+
\frac{c_1}{2\lambda(\lambda-1) R^2}.
$$
The restriction on $R$ in the proposition statement is such that
$$
\frac{c_1}{2\lambda (\lambda-1) R^2} \leq \frac12,
$$
from which it follows that
$$
r''(t)\leq -\frac12.
$$
The remainder of the argument is the same as in the proof of Prop.
\ref{P:blowup-time}.
\end{proof}

\section{Variational characterization of the ground state}
\label{S:varQ}

For now, write $u=u(x)$ (time dependence plays no role) in what
follows in this section.  The goal of this section is a variational
characterization of the ground state $Q$ stated below as Prop.
\ref{P:varQ}.  For the proof we will just show how it follows from
scaling, the bounds depicted in Figure \ref{F:dichotomy}, and an
existing characterization of $Q$ appearing in Lions \cite[Theorem
I.2]{L84}.  Prop. \ref{P:varQ} will be one of the main ingredients
in our treatment of the ``near boundary case'' in \S
\ref{S:near-boundary}.

\begin{proposition}[Variational characterization of the ground state]
 \label{P:varQ}
There exists a function $\epsilon(\rho)$ with $\epsilon(\rho)\to 0$
as $\rho\to 0$ such that the following holds: Suppose there is
$\lambda>0$ such that
\begin{equation}
\label{E:MEclose} \left| \frac{M[u]E[u]}{M[Q]E[Q]} -
(3\lambda^2-2\lambda^3) \right| \leq \rho \lambda^3\, ,
\end{equation}
and
\begin{equation}
\label{E:L2Gclose} \left| \frac{\|u\|_{L^2}\|\nabla
u\|_{L^2}}{\|Q\|_{L^2}\|\nabla Q\|_{L^2}} - \lambda \right| \leq
\rho \begin{cases} \lambda^2 & \text{if }\lambda \leq 1 \\ \lambda &
\text{if } \lambda \geq 1 \end{cases} \,.
\end{equation}
Then there exists $\theta\in \mathbb{R}$ and $x_0\in \mathbb{R}^3$
such that
\begin{equation}
\label{E:L2char2} \|u(x) -
e^{i\theta}\lambda^{3/2}\beta^{-1}Q(\lambda( \beta^{-1} x - x_0))
\|_{L_x^2} \leq \beta^{1/2}\epsilon(\rho)
\end{equation}
and
\begin{equation}
\label{E:H1char2} \left\|\nabla \big[ u(x) -
e^{i\theta}\lambda^{3/2}\beta^{-1}Q(\lambda( \beta^{-1}x - x_0))
\big] \right\|_{L_x^2} \leq \lambda \beta^{-1/2}\epsilon(\rho),
\end{equation}
where $\beta = M[u]/M[Q]$.
\end{proposition}

\begin{remark}
Note that the right-hand side bounds in \eqref{E:MEclose} and
\eqref{E:L2Gclose} do not depend on the mass.  Moreover, the
conclusion \eqref{E:L2char2} and \eqref{E:H1char2} could be replaced
with the weaker statement
$$
\|u(x) - e^{i\theta} \lambda^{3/2}\beta^{-1}Q(\lambda( \beta^{-1} x
- x_0)) \|_{L_x^2}\left\|\nabla \big[ u(x) -
e^{i\theta}\lambda^{3/2}\beta^{-1}Q(\lambda( \beta^{-1}x - x_0))
\big] \right\|_{L_x^2} \leq \epsilon(\rho),
$$
which also has a right-hand side independent of the mass.
\end{remark}

\begin{remark}
\label{R:rescale} Define $v(x) = \beta u(\beta x)$ and note that
$M[v] = \beta^{-1} M[u] = M[Q]$. Now we can restate Proposition
\ref{P:varQ} as follows:

Suppose $\|v\|_{L^2}=\|Q\|_{L^2}$ and there is $\lambda>0$ such that
\begin{equation}
 \label{E:energy-close}
\left| \frac{E[v]}{E[Q]} - (3\lambda^2-2\lambda^3) \right| \leq
\lambda^3 \rho  \, ,
\end{equation}
and
\begin{equation}
 \label{E:gradient-close}
\left| \frac{\|\nabla v\|_{L^2}}{\|\nabla Q\|_{L^2}} - \lambda
\right|
\leq  \rho \begin{cases} \lambda^2 & \text{if }\lambda \leq 1 \\
\lambda & \text{if }\lambda \geq 1 \end{cases}\,.
\end{equation}
Then there exists $\theta \in \mathbb{R}$ and $x_0\in \mathbb{R}^3$
such that
\begin{equation}
 \label{E:L2char}
\|v - e^{i\theta}\lambda^{3/2}Q(\lambda( \bullet - x_0)) \|_{L^2}
\leq \epsilon(\rho)
\end{equation}
and
\begin{equation}
 \label{E:H1char}
\left\|\nabla \big[ v - e^{i\theta}\lambda^{3/2}Q(\lambda( \bullet -
x_0)) \big] \right\|_{L^2} \leq \lambda \,\epsilon(\rho).
\end{equation}
In fact, Prop. \ref{P:varQ} is equivalent to the above scaled
statement.
\end{remark}

We first restate the result from Lions \cite[Theorem I.2]{L84} below
as Prop. \ref{P:ccQ} and then show how the proof of Prop.
\ref{P:varQ} follows from Prop. \ref{P:ccQ}.

\begin{proposition}{\cite[Theorem I.2]{L84}}
 \label{P:ccQ}
There exists a function $\epsilon(\rho)$, defined for small $\rho>0$
and such that $\lim\limits_{\rho\rightarrow 0} \epsilon(\rho)=0$,
such that for all $u \in H^1$ with
\begin{equation}
 \label{E:close}
\Big|\|u\|_4-\|Q\|_4\Big|+\Big|\|u\|_2 - \|Q\|_2\Big| +
\Big|\|\nabla u\|_2-\|\nabla Q\|_2\Big|\leq \rho,
\end{equation}
there exist  %
$\ds \theta_0 \in \cR$ and $x_0 \in \cR^3$ such that
\begin{equation}
 \label{E:Qclose}
\left\| u - e^{i\theta_0}Q(\bullet -x_0)\right\|_{H^1} \leq
\epsilon(\rho).
\end{equation}
\end{proposition}

\begin{proof}[Proof of Prop. \ref{P:varQ}]
We prove  Remark \ref{R:rescale} which is equivalent to Prop.
\ref{P:varQ} by rescaling off the mass. Set $\tilde u(x) =
\lambda^{-3/2} v(\lambda^{-1}x)$.  Then \eqref{E:gradient-close}
implies
\begin{equation}
 \label{E:renormgrad-close}
\left| \frac{\|\nabla \tilde{u} \|_{L^2}}{\|\nabla Q\|_{L^2}} - 1
\right| \leq  \rho.
 \end{equation}
Next, by \eqref{E:energy-close} and \eqref{E:gradient-close} we have
\begin{align*}
2\left| \frac{\|v\|_{L^4}^4}{\|Q\|_{L^4}^4} - \lambda^3 \right|
&\leq \left| \frac{E[v]}{E[Q]} - (2\lambda^3-3\lambda^2) \right|
+ 3\left| \frac{\|\nabla v\|_{L^2}^2}{\|\nabla Q\|_{L^2}^2} - \lambda^2 \right| \\
&\leq \rho \left( \lambda^3 + 3 \begin{cases} \lambda^3 & \text{if }\lambda \leq 1 \\
\lambda^2 & \text{if } \lambda \geq 1 \end{cases} \right) \\
&\leq 4\lambda^3 \rho.
\end{align*}
Thus, in terms of $\tilde u$, we obtain
\begin{equation}
 \label{E:L4tilde}
\left| \frac{\| \tilde u\|_{L^4}^4}{\|Q\|_{L^4}^4} - 1 \right| \leq
2 \rho.
\end{equation}
Hence, \eqref{E:renormgrad-close} and \eqref{E:L4tilde} imply the
condition \eqref{E:close} for $\tilde u$ (the factors in front of
$\rho$ in both inequalities can be inconsequentially incorporated
into $\rho$), and by Proposition \ref{P:ccQ}, there exist $\theta
\in \cR$ and $x_0 \in \cR^3$ such that \eqref{E:Qclose} holds for
$\tilde u$. Rescaling back to $v$, we obtain exactly
\eqref{E:L2char} and \eqref{E:H1char}.
\end{proof}

\section{Near-boundary case}
\label{S:near-boundary}

We know by Prop. \ref{P:dichotomy} that if $M[u]=M[Q]$ and
$E[u]/E[Q]=3\lambda^2-2\lambda^3$ for some $\lambda>1$ and $\|\nabla
u_0\|_{L^2}/\|\nabla Q\|_{L^2} \geq 1$, then $\|\nabla u(t)
\|_{L^2}/\|\nabla Q\|_{L^2} \geq \lambda$ for all $t$. The next
result says that  $\|\nabla u(t)\|_{L^2}/\|\nabla Q\|_{L^2}$ cannot,
globally  in time, remain near $\lambda$.

\begin{proposition}[Near boundary case]
 \label{P:nbc}
Let $\lambda_0>1$. There exists $\rho_0= \rho_0(\lambda_0)>0$ (with
the property that $\rho_0\to 0$ as $\lambda_0\searrow 1$) such that
for any $\lambda\geq \lambda_0$, the following holds:  There does
\textsc{not} exist a solution $u(t)$ of NLS with $P[u]=0$ satisfying
$\|u\|_{L^2} = \|Q\|_{L^2}$,
\begin{equation}
 \label{E:nbc1}
\frac{E[u]}{E[Q]} = 3\lambda^2-2\lambda^3 \, ,
\end{equation}
and
\begin{equation}
 \label{E:nbc2}
\lambda \leq \frac{\|\nabla u(t)\|_{L^2}}{\|\nabla Q\|_{L^2}} \leq
\lambda(1+\rho_0) \quad \text{for all }t\geq 0\, .
\end{equation}
\end{proposition}

Of course, the assertion is equivalent to:  For every solution
$u(t)$ of NLS  with $P[u]=0$ satisfying $\|u\|_{L^2} = \|Q\|_{L^2}$,
$$
\frac{E[u]}{E[Q]} = 3\lambda^2-2\lambda^3 \, ,
$$
and
$$
\lambda \leq \frac{\|\nabla u(t)\|_{L^2}}{\|\nabla Q\|_{L^2}} \quad
\text{for all }t\geq 0,
$$
there exists a time $t_0 \geq 0$ such that
$$
\frac{\|\nabla u(t_0)\|_{L^2}}{\|\nabla Q\|_{L^2}} \geq
\lambda(1+\rho_0),
$$
equivalently, there exists a sequence $t_n\to +\infty$ such that
$$
\frac{\|\nabla u(t_n)\|_{L^2}}{\|\nabla Q\|_{L^2}} \geq
\lambda(1+\rho_0)
$$
for all $n$. This seemingly stronger statement is seen to be
equivalent by ``resetting'' the initial time $\tilde u(t) = u(t -
t_0-1)$ for $t\geq 0$.

We shall need a version of Lemma 5.1 from \cite{DHR}.

\begin{lemma}
\label{L:center-control}
 Suppose that $u(t)$ with $P[u]=0$ solving \eqref{E:NLS} satisfies, for all $t$,
\begin{equation}
 \label{E:apx-close}
\|u(t) - e^{i\theta(t)}Q(\bullet - x(t))\|_{H^1} \leq \epsilon
\end{equation}
for some continuous functions $\theta(t)$ and $x(t)$.  Then
$$
\frac{|x(t)|}{t} \lesssim \epsilon^2 \quad \text{as} \quad
t\to+\infty.
$$
\end{lemma}

The proof of this lemma is very similar to that of Lemma 5.1 in
\cite{DHR}.  For the reader's convenience, we carry it out in
Appendix \ref{S:pf-center-control}.

\begin{proof}[Proof of Prop. \ref{P:nbc}]
The constants $c_j$ we introduce below are absolute constants. To
the contrary, suppose that $u(t)$ is a solution of the type
described in the proposition statement, i.e., $\|u\|_{L^2} =
\|Q\|_{L^2}$, $E[u]/E[Q] = 3\lambda^2-2\lambda^3$ and
\begin{equation}
\label{E:grad-up-bd}
 \lambda \leq \frac{\|\nabla u(t)\|_{L^2}}{\|\nabla Q\|_{L^2}} \leq
\lambda(1+\rho_0) \quad \text{for all }t\geq 0 \,.
\end{equation}
Since $\|\nabla u(t)\|_{L^2}^2 \geq \lambda^2 \|\nabla Q\|^2_{L^2} =
6 \lambda^2 E[Q]$, we have
$$
24E[u] - 4\|\nabla u(t)\|_{L^2}^2 \leq -48E[Q]\lambda^2(\lambda-1).
$$
By Prop. \ref{P:varQ}, there exist functions $x(t)$ and $\theta(t)$
such that
\begin{equation}
 \label{E:L2x(t)}
\|u(t) - e^{i\theta(t)}\lambda^{3/2}Q(\lambda( \bullet + x(t)))
\|_{L^2} \leq \epsilon(\rho),
\end{equation}
\begin{equation}
 \label{E:dotH1x(t)}
\|u(t) - e^{i\theta(t)}\lambda^{3/2}Q(\lambda( \bullet + x(t)))
\|_{\dot{H}^1} \leq \lambda \epsilon(\rho).
\end{equation}
By continuity of the $u(t)$ flow, we may assume that $\theta(t)$ and
$x(t)$ are continuous. Let
$$
R(T) = \max \left(\max_{0\leq t\leq T} |x(t)|, \;
\log\epsilon(\rho)^{-1} \right) \,.
$$
Fix $T>0$.  Take $R=2R(T)$ in the local virial identity
\eqref{E:locvar''}. By \eqref{E:L2x(t)}-\eqref{E:dotH1x(t)}, there
exists $c_2>0$ such that
$$
|A_R(u(t))| \leq \tfrac12 c_2 \lambda^2
\left(\epsilon(\rho)+e^{-R(T)}\right)^2 \leq c_2 \lambda^2
\epsilon(\rho)^2.
$$
Consequently, by taking $\rho_0$ small enough, we can make
$\epsilon(\rho)$ small enough so that for all $0\leq t\leq T$,
$$
z_R''(t) \leq -24E[Q]\lambda^2(\lambda-1).
$$
(Note that here, the closer $\lambda>1$ is to $1$, the smaller
$\rho_0$ needs to be taken.)  By integrating in time over $[0,T]$
twice, we obtain that
$$
\frac{z_R(T)}{T^2} \leq \frac{z_R(0)}{T^2} + \frac{z_R'(0)}{T} - 12
E[Q]\lambda^2(1-\lambda)\,.
$$
We have
$$
|z_R(0)| \leq c_3 R^2\|u_0\|_{L^2}^2 = c_3\|Q\|_{L^2}^2 R^2,
$$
and
$$
|z_R'(0)| \leq c_3 R\|u_0\|_{L^2}\|\nabla u_0\|_{L^2}\leq c_3
\|Q\|_{L^2} \|\nabla Q\|_{L^2}(1+\rho_0) R,
$$
and as a result
$$
z_{2R(T)}(T) \leq c_4\left(\frac{R(T)^2}{T^2} +
\frac{R(T)}{T}\right) - 12E[Q]\lambda^2(\lambda-1).
$$
By taking $T$ sufficiently large and applying Lemma
\ref{L:center-control}, we obtain
$$
0\leq z_{2R(T)}(T) \leq c_4\epsilon(\rho)^2 -   12E[Q]\lambda^2(\lambda-1)<0$$
provided $\rho_0$ is selected small enough so that
$c_4\epsilon(\rho)^2 \leq 6E[Q]\lambda^2(\lambda-1)$.  Note this
selection of $\rho_0$ is independent of $T$. This is a
contradiction.
\end{proof}

\section{Profile decomposition}
\label{S:pd}

Let us recall the Keraani-type profile decomposition lemma and some
associated results from \cite{HR}, \cite{DHR}.  We first need to
review the Strichartz norm notation from \cite{HR}.

We say that
$(q,r)$ is $\dot H^s$ Strichartz admissible (in 3d) if
$$
\frac2q+\frac3r=\frac32-s.
$$
Let
$$
\|u\|_{S(L^2)} = \sup_{\substack{(q,r)\; L^2 \text{ admissible} \\
2\leq r \leq 6, \;  2\leq q \leq \infty}} \|u\|_{L_t^qL_x^r}.
$$
Define
$$
\|u\|_{S(\dot H^{1/2})} = \sup_{\substack{(q,r)\; \dot H^{1/2} \text{ admissible} \\
3\leq  r \leq 6^-, \; 4^+ \leq q \leq \infty}} \|u\|_{L_t^qL_x^r} \,
,
$$
where $6^-$ is an arbitrarily preselected and fixed number $<6$;
similarly for $4^+$.
Now we consider dual Strichartz norms. Let
$$
\|u\|_{S'(L^2)} = \inf_{\substack{(q,r)\; L^2 \text{ admissible} \\
2\leq q \leq \infty, \; 2\leq r \leq 6}} \|u\|_{L_t^{q'}L_x^{r'}},
$$
where $(q',r')$ is the H\"older dual to $(q,r)$.  Also define
$$
\|u\|_{S'(\dot H^{-1/2})} = \inf_{\substack{(q,r)\; \dot H^{-1/2} \text{ admissible} \\
\frac{4}{3}^+ \leq q \leq 2^-, \; 3^+ \leq r \leq 6^-}} \|u\|_{L_t^{q'}L_x^{r'}} \, .
$$

We extend our notation $S(\dot H^s)$, $S'(\dot H^s)$ as follows: If
a time interval is not specified (that is, if we just write $S(\dot
H^s)$, $S'(\dot H^s)$), then the $t$-norm is evaluated over
$(-\infty,+\infty)$.  To indicate a restriction to a time
subinterval $I\subset (-\infty,+\infty)$, we will write $S(\dot
H^s;I)$ or $S'(\dot H^s; I)$.  We shall also use the notation $\nls(t)$ to indicate the nonlinear flow map associated to \eqref{E:NLS}.  

The following proposition incorporates results from our earlier
papers.  The basic form of the (linear) profile decomposition is
proved in \cite[Lemma 5.2]{HR}, \cite[Lemma 2.1]{DHR} (and the proof
given there was modeled on a similar result of Keraani \cite{K}).
The proof of \eqref{E:energy-Pythag} is given in \cite[Lemma
2.3]{DHR} and the method of replacing linear flows by nonlinear
flows appears as part of \cite[Prop. 5.4, 5.5]{HR}.

\begin{proposition}
 \label{P:pd}
Suppose that $\phi_n=\phi_n(x)$ is a bounded sequence in $H^1$.
There exist a subsequence of $\phi_n$ (still denoted $\phi_n$),
profiles $\psi^j$ in $H^1$, and parameters $x_n^j$, $t_n^j$ so that
for each $M$,
$$
\phi_n = \sum_{j=1}^M \nls(-t_n^j)\,\psi^j(\bullet-x_n^j) + W_n^M,
$$
where (as $n \to \infty$):
\begin{itemize}
\item
For each $j$, either $t_n^j=0$ \footnote{This is done by passing to
another subsequence in $n$ and adjusting the profiles $\psi^j$; see
also comment in Step 1 of the proof \cite[Lemma 2.3]{DHR}.},
$t_n^j\to +\infty$, or $t_n^j \to -\infty$.
\item
If $t_n^j \to +\infty$, then $\|\nls(-t)\,\psi^j\|_{S(\dot H^{1/2};
[0,+\infty))} <\infty$ and if $t_n^j \to -\infty$, then
$\|\nls(-t)\,\psi^j\|_{S(\dot H^{1/2};(-\infty,0])}< \infty$.
\item
For $j\neq k$,
$$
|t_n^j-t_n^k| + |x_n^j-x_n^k| \to +\infty.
$$
\item
$\nls(t)W_n^M$ is global for $M$ large enough with ~$\ds
\left(\lim_n \|\nls(t)\,W_n^M\|_{S(\dot H^{1/2})} \right) \to 0$ as
$M\to \infty$.  (Note: we do not claim that $\ds \lim_n
\|W_n^M\|_{H^1}\to 0$.)
\end{itemize}

We also have the $\dot H^s$ Pythagorean decomposition: For fixed $M$
and $0 \leq s \leq 1$, we have
\begin{equation}
\label{E:Hs-Pythag} \|\phi_n\|_{\dot H^s}^2 = \sum_{j=1}^M
\|\nls(-t_n^j)\,\psi^j\|_{\dot H^s}^2 + \|W_n^M\|_{\dot H^s}^2 +
o_n(1).
\end{equation}
We also have the energy Pythagorean decomposition \footnote{By
energy conservation $E[\psi^j] = E[\nls(-t_n^j) \, \psi^j]$.}:
\begin{equation}
\label{E:energy-Pythag} E[\phi_n] = \sum_{j=1}^M E[\psi^j] +
E[W_n^M] + o_n(1).
\end{equation}
\end{proposition}
A similar statement to \eqref{E:energy-Pythag} was proved in
\cite[Lemma 2.3]{DHR} for the linear flows $e^{-it_n^j\Delta}\psi^j$
by establishing the $L^4$
orthogonal decomposition, and implicitly (by the existence of wave
operators and the long-term perturbation argument) for the nonlinear
flow:
\begin{equation}
 \label{E:L4decomp}
\|\phi_n\|_{L^4}^4 = \sum_{j=1}^M \|\nls(-t_n^j)\,\psi^j\|_{L^4}^4 +
\|W_n^M\|_{L^4}^4 + o_n(1),
\end{equation}
and thus, the energy Pythagorean decomposition
\eqref{E:energy-Pythag} follows.

The next lemma is taken from \cite[Prop. 2.3]{HR} (the statement is
slightly different, but the proof given there actually establishes
the statement given below):

\begin{lemma}[perturbation theory]
\label{L:pt} For each $A\gg 1$, there exists
$\epsilon_0=\epsilon_0(A)\ll 1$ and $c=c(A)$ such that the following
holds.  Fix $T>0$.  Let $u=u(x,t)\in L_{[0,T]}^\infty H_x^1$ solve
$$
i\partial_t u + \Delta u + |u|^2u=0
$$
on $[0,T]$. Let $\tilde u(x,t)\in L_{[0,T]}^\infty H_x^1$ and define
$$
e=i\partial_t\tilde u + \Delta \tilde u + |\tilde u|^2\tilde u.
$$
For each $\epsilon \leq \epsilon_0$, if
$$
\|\tilde u \|_{S(\dot H^{1/2}; [0,T])} \leq A,
\quad \|e\|_{S'(\dot H^{1/2};[0,T])} \leq \epsilon, \quad \text{and} \quad
\| e^{it\Delta}(u(0)-\tilde u(0))\|_{S(\dot H^{1/2};[0,T])} \leq \epsilon \, ,
$$
then
$$
\|u-\tilde u\|_{S(\dot H^{1/2};[0,T])} \leq c(A)\epsilon \,.
$$
\end{lemma}

We remark that $T$ does not actually enter into the parameter
dependence in any way: $\epsilon_0$ depends only on $A$, not on $T$.
In fact, in \cite[Prop. 2.3]{HR}, $T=+\infty$.  Now, in our
application below, it will turn out that $A=A(T)$, so ultimately
there will be dependence upon $T$, but it is only through $A$.

The equation \eqref{E:Hs-Pythag} gives $\dot H^1$ asymptotic
orthogonality at $t=0$, but we will need to extend this to the NLS
flow for $0\leq t \leq T$.  This is the subject of the next lemma,
which does not appear in our previous papers.

\begin{lemma}[~$\dot H^1$ Pythagorean decomposition along the $\nls$ flow]
 \label{L:Pythag-flow}
Suppose (as in Prop. \ref{P:pd}) $\phi_n$ is a bounded sequence in
$H^1$. Fix any time $0<T<\infty$. Suppose that $u_n(t)\equiv
\nls(t)\phi_n$ exists up to time $T$ for all $n$ and
$$
\lim_{n \to \infty} \|\nabla u_n(t)\|_{L_{[0,T]}^\infty
L_x^2}<\infty.
$$
Let $W_n^M(t)\equiv \nls(t)W_n^M$ (which we know is global and, in
fact, scattering).  Then, for all $j$, $v^j(t) \equiv
\nls(t)\,\psi^j$ exist up to time $T$ and for all  $t\in [0,T]$,
\begin{equation}
 \label{E:orth_flow}
\|\nabla u_n(t)\|_{L^2}^2 = \sum_{j=1}^M \|\nabla
v^j(t-t_n^j)\|_{L^2}^2 + \|\nabla W_n^M(t)\|_{L_x^2}^2 + o_n(1).
\end{equation}
Here, $o_n(1)\to 0$ uniformly on $0\leq t\leq T$.
\end{lemma}
\begin{proof}
Let $M_0$ be such that for $M\geq M_0$, we have
$\|\nls(t) \,W_n^M\|_{S(\dot H^{1/2})} \leq \delta_{\text{sd}}$
($\delta_{\text{sd}}$ is the small data scattering threshold defined
in \cite{HR}).  Reorder the first $M_0$ profiles and introduce an
index $M_2$, $0\leq M_2\leq M$, so that
\begin{enumerate}
\item
For each $1\leq j \leq M_2$, we have $t_n^j=0$.  If $M_2=0$, that
means there are no $j$ in this category.
\item
For each $M_2+1\leq j\leq M_0$, we have $|t_n^j|\to +\infty$.  If
$M_2=M_0$, that means there are no $j$ in this category.
\end{enumerate}
We then know from the profile construction that the $v^j(t)$ for $j>
M_0$ are scattering (in both time directions).  It follows from
Prop. \ref{P:pd} that for fixed $T$ and $M_2+1\leq j\leq M_0$, we
have $\|v^j(t-t_n^j)\|_{S(\dot H^{1/2}; [0,T])} \to 0$ as $n\to
+\infty$.  Indeed, consider the case $t_n^j \to +\infty$ and
$\|v^j(-t)\|_{S(\dot H^{1/2}; [0,+\infty))}<\infty$.  Then for
$q<\infty$, it is immediate from dominated convergence that
$\|v^j(-t)\|_{L_{[0,+\infty)}^q L_x^r}<\infty$ implies
$\|v^j(t-t_n^j)\|_{L_{[0,T]}^q L_x^r} \to 0$.  Since $v^j$ is
constructed in Prop. \ref{P:pd} via the existence of wave operators
\cite[Prop. 4.6]{HR} to converge in $H^1$ to a linear flow at
$-\infty$, it follows from the $L_x^3$ decay of the linear flow that
$\|v^j(t-t_n^j)\|_{L_{[0,T]}^\infty L_x^3} \to 0$.

Let $B=\max(1,\lim_n \|\nabla u_n(t)\|_{L_{[0,T]}^\infty
L_x^2})<\infty$.  For each $1\leq j \leq M_2$, define $T^j$ to be
the maximal forward time $\leq T$ on which $\|\nabla
v^j\|_{L_{[0,T^j]}^\infty L_x^2} \leq 2B$.  Let $\tilde
T=\min_{1\leq j\leq M_2} T^j$ (if $M_2=0$, then just take $\tilde
T=T$.)  We will begin by proving that \eqref{E:orth_flow} holds for
$T=\tilde T$.  It will then follow from \eqref{E:orth_flow} that for
each $1\leq j\leq M_2$, we have $T^j =T$, and hence, $\tilde T=T$.
Thus, for the remainder of the proof, we work on $[0,\tilde T]$.
For each $1\leq j\leq M_2$, we have
\begin{align*}
\|v^j(t)\|_{S(\dot H^{1/2};[0,\tilde T])}
&\lesssim \|v^j\|_{L_{[0,\tilde T]}^\infty L_x^3} + \|v^j\|_{L_{[0,\tilde T]}^4 L_x^6} \\
&\lesssim \|v^j\|_{L_{[0,\tilde T]}^\infty L_x^2}^{1/2} \|
\nabla v^j\|_{L^\infty_{[0,\tilde T]}L_x^2}^{1/2} + \tilde T^{1/4}
\| \nabla v^j\|_{L^\infty_{[0,\tilde T]}L_x^2}\\
&\lesssim \la \tilde T^{1/4}\ra B,
\end{align*}
where we have used that $\|v^j\|_{L^\infty_{[0,\tilde T]} L_x^2} =
\|\psi^j\|_{L_x^2} \leq \lim_n \|\phi_n\|_{L^2}$ by
\eqref{E:Hs-Pythag} with $s=0$.

Let
$$
\tilde u_n(x,t) = \sum_{j=1}^M v^j(x-x_n^j,t-t_n^j).
$$
Of course, $\tilde u_n$ also depends upon $M$ but we suppress this
dependence from the notation.  Also, let
$$
e_n = i\partial_t \tilde u_n + \Delta \tilde u_n + |\tilde
u_n|^2\tilde u_n.
$$

We now outline a series of claims, which we do not prove here since
the proofs closely follow the proof of \cite[Prop. 5.4]{HR}.

\medskip

\noindent\emph{Claim 1}. There exists $A=A(\tilde T)$ (independent
of $M$ but dependent on $\tilde T$) such that for all $M> M_0$,
there exists $n_0=n_0(M)$ such that for all $n> n_0$,
$$
\|\tilde u_n\|_{S(\dot H^{1/2};[0,\tilde T])} \leq A \, .
$$

\medskip

\noindent\emph{Claim 2}.  For each $M>M_0$ and $\epsilon>0$, there
exists $n_1=n_1(M,\epsilon)$ such that for $n>n_1$,
$$
\|e_n\|_{L_{[0,\tilde T]}^{10/3} L_x^{5/4}} \leq \epsilon \, .
$$

\medskip

\noindent\emph{Remark 3}. Note that since $u_n(0)-\tilde u_n(0) =
W_n^M$, there exists $M'=M'(\epsilon)$ sufficiently large so that
for each $M>M'$ there exists $n_2=n_2(M)$ such that $n>n_2$ implies
$$
\|e^{it\Delta}(u_n(0)-\tilde u_n(0))\|_{S(\dot H^{1/2};[0,\tilde
T])} \leq \epsilon \, .
$$

\medskip

Recall we are given $\tilde T$, and thus, by Claim 1, there is a
large number $A(\tilde T)$. Then the statement of Lemma \ref{L:pt}
gives us $\epsilon_0=\epsilon_0(A)$.  Now select an arbitrary
$\epsilon \leq \epsilon_0$, and obtain from Remark 3 an index
$M'=M'(\epsilon)$. Now select an arbitrary $M>M'$.  Set
$n'=\max(n_0,n_1,n_2)$. Then we conclude from Claims 1-2, Remark 3,
and Lemma \ref{L:pt}, that for $n>n'(M,\epsilon)$,
\begin{equation}
 \label{E:u-tildeu-close}
\|u_n-\tilde u_n\|_{S(\dot H^{1/2};[0,\tilde T])} \leq c(\tilde T)
\,  \epsilon \,,
\end{equation}
where $c=c(A)=c(\tilde T)$.

Now we prove \eqref{E:orth_flow} on $[0,\tilde T]$.  We know that
for each $1\leq j \leq M_2$, we have $\|\nabla
v^j(t)\|_{L_{[0,\tilde T]}^\infty L_x^2}\leq 2B$.  Let us discuss
$j\geq M_2+1$.  As we've noted, $\|v^j(t-t_n^j)\|_{S(\dot
H^{1/2};[0,\tilde T])}\to 0$ as $n\to +\infty$.  By the Strichartz
estimates, $\|\nabla v^j(t-t_n^j)\|_{L_{[0,\tilde T]}^\infty L_x^2}
\lesssim \|\nabla v^j(-t_n^j)\|_{L_x^2}$.
By the pairwise divergence of parameters,
\begin{align*}
\|\nabla \tilde u_n(t)\|_{L_{[0,\tilde T]}^\infty L_x^2}^2
&= \sum_{j=1}^{M_2} \|\nabla v^j(t)\|_{L_{[0,\tilde T]}^\infty L_x^2}^2
+ \sum_{j=M_2+1}^M \|\nabla v^j(t-t_n^j)\|_{L_{[0,\tilde T]}^\infty L_x^2}^2+o_n(1) \\
&\lesssim M_2B^2+ \sum_{j=M_2+1}^M \|\nabla \nls(-t_n^j)\,\psi^j\|_{L_x^2}^2 + o_n(1) \\
&\leq M_2B^2+ \|\nabla \phi_n\|_{L_x^2}^2 + o_n(1)\\
&\leq M_2B^2+B^2+o_n(1).
\end{align*}

From \eqref{E:u-tildeu-close}, we conclude that
\begin{align*}
\|u_n - \tilde u_n \|_{L_{[0,\tilde T]}^\infty L_x^4}
&\lesssim \|u_n - \tilde u_n\|_{L_{[0,\tilde T]}^\infty L_x^3}^{1/2}
\|\nabla (u_n-\tilde u_n)\|_{L_{[0,\tilde T]}^\infty L_x^2}^{1/2}\\
&\leq c(\tilde T)^{1/2}(M_2B^2+2B^2+o_n(1))^{1/4}\epsilon^{1/2}.
\end{align*}

An argument similar to the proof of \eqref{E:L4decomp} now
establishes that, for each $t\in [0,\tilde T]$,
\begin{equation}
 \label{E:L4flow}
\|u_n(t)\|_{L^4}^4 = \sum_{j=1}^M \|v^j(t-t_n^j)\|_{L^4}^4 +
\|W_n^M(t)\|_{L^4}^4 + o_n(1).
\end{equation}
By \eqref{E:energy-Pythag} and energy conservation
($E[\psi^j]=E[v^j(t-t_n^j)]$, etc.), we have
\begin{equation}
 \label{E:en-at-t}
E[u_n(t)] = \sum_{j=1}^M E[v^j(t-t_n^j)] + E[W_n^M(t)] + o_n(1).
\end{equation}
Combining \eqref{E:L4flow} and \eqref{E:en-at-t} gives \eqref{E:orth_flow}.
\end{proof}

\begin{lemma}[profile reordering]
 \label{L:exist-nonscat}
Suppose that $\phi_n=\phi_n(x)$ is an $H^1$ bounded sequence to
which we apply the Prop. \ref{P:pd} out to a given $M$. Let
$\lambda_0>1$.  Suppose that $M[\phi_n]=M[Q]$, $E[\phi_n]/E[Q] =
3\lambda_n^2-2\lambda_n^3$ with $\lambda_n\geq \lambda_0>1$ and
$\|\nabla \phi_n\|_{L^2}/\|\nabla Q\|_{L^2}\geq \lambda_n$ for each
$n$.  Then, the profiles can be reordered so that there exists
$1\leq M_1\leq M_2\leq M$ and
\begin{enumerate}
\item
 \label{I:type1}
For each $1\leq j\leq M_1$, we have $t_n^j=0$ and $v^j(t)\equiv
\nls(t)\psi^j$ does not scatter as $t\to +\infty$.  (In particular,
we are asserting the existence of at least one $j$ that falls into
this category.)

\item
 \label{I:type2}
For each $M_1+1\leq j \leq M_2$, we have $t_n^j=0$ and $v^j(t)$
scatters as $t\to +\infty$. (If $M_1=M_2$, there are no $j$ with
this property.)

\item
 \label{I:type3}
For each $M_2+1\leq j\leq M$, we have that $|t_n^j|\to +\infty$. (If
$M_2=M$, there are no $j$ with this property.)
\end{enumerate}
\end{lemma}

\begin{proof}
We first prove that there exists at least one $j$ such that $t_n^j$
converges as $n\to +\infty$. Indeed, it follows that
\begin{align*}
\frac{\|\phi_n\|_{L^4}^4}{\|Q\|_{L^4}^4}
& =-\frac{E[\phi_n]}{2E[Q]}+ \frac{3\|\phi_n\|_{L^2}^2}{2\|\nabla Q\|_{L^2}^2}\\
& \geq -\frac12(3\lambda_n^2-2\lambda_n^3) + \frac32\lambda_n^2\\
& = \lambda_n^3 \geq \lambda_0^3>1.
\end{align*}
Now if $j$ is such that $|t_n^j|\to \infty$, then
$\|\nls(-t_n^j)\psi^j\|_{L^4} \to 0$.  The claim now follows from
\eqref{E:L4decomp}.  Note that if $j$ is such that $t_n^j$ converges
as $n\to +\infty$, then we might as well WLOG assume that $t_n^j=0$
(see also footnote 7).

Reorder the profiles $\psi^j$ so that for $1\leq j \leq M_2$, we
have $t_n^j=0$, and for $M_2+1\leq j\leq M$, we have $|t_n^j|\to
\infty$. It only remains to show that there exists one $j$, $1\leq j
\leq M_2$ such that $v^j(t)$ is nonscattering.  If not, then for all
$1\leq j\leq M_2$, we have that all $v^j$ are scattering, and thus,
$\lim_{t\to +\infty} \|v^j(t)\|_{L^4} =0$.  Let $t_0$ be large
enough so that, for all $1\leq j\leq M_2$, we have
$\|v^j(t_0)\|_{L^4}^4 \leq \epsilon/M_2$.  By the $L^4$
orthogonality \eqref{E:L4flow} along the $\nls$ flow, we have
\begin{align*}
\lambda_0^3\|Q\|_{L^4}^4 &\leq \|u_n(t_0)\|_{L^4}^4 \\
&= \sum_{j=1}^{M_2} \|v^j(t_0)\|_{L^4}^4 + \sum_{j=M_2+1}^M
\|v^j(t_0-t_n^j)\|_{L_x^4}^4 + \|W_n^M(t_0)\|_{L_x^4}^4 + o_n(1).
\intertext{As $n\to +\infty$, we have $\sum_{j=M_2+1}^M
\|v^j(t_0-t_n^j)\|_{L_x^4}^4\to 0$, and thus, the last line} &\leq
\epsilon + \|W_n^M(t_0)\|_{L_x^4}^4 + o_n(1).
\end{align*}
This gives a contradiction.
\end{proof}

\section{Outline of the inductive argument}
\label{S:outline-induction}

Having developed several preliminaries in \S
\ref{S:prelim}--\ref{S:pd}, we now begin the proof of Theorem
\ref{T:main}.

Consider the following statement:

\begin{definition}
Let $\lambda >1$.  We say that $\egb(\lambda, \sigma)$ {\rm holds}
if there exists a solution $u(t)$ to NLS such that
$$
M[u]=M[Q], \quad \frac{E[u]}{E[Q]}=3\lambda^2-2\lambda^3 \,,
$$
and
$$
\lambda \leq \frac{\|\nabla u(t)\|_{L^2}}{\|\nabla Q\|_{L^2}} \leq
\sigma \quad \text{for all }t\geq 0.
$$
\end{definition}

$\egb(\lambda,\sigma)$ can be read ``there \underline{e}xist
solutions at energy $3\lambda^2-2\lambda^3$ \underline{g}lobally
\underline{b}ounded by $\sigma$.''

By Prop. \ref{P:nbc}, $\egb(\lambda, \lambda(1+\rho_0(\lambda_0)))$
is \emph{false} for all $\lambda \geq \lambda_0>1$.

Note that the statement ``$\egb(\lambda,\sigma)$ is false'' is
equivalent to the statement:  For every solution $u(t)$ to NLS such
that $M[u]=M[Q]$ and $E[u]/E[Q]=3\lambda^2-2\lambda^3$ such that
$\lambda \leq \|\nabla u(t)\|_{L^2}/\|\nabla Q\|_{L^2}$ for all $t$,
there exists a time $t_0\geq 0$ such that $\|\nabla
u(t_0)\|_{L^2}/\|\nabla Q\|_{L^2} \geq \sigma$.  (In fact, there
exists a sequence $t_n \to +\infty$ such that $\|\nabla
u(t_n)\|_{L^2}/\|\nabla Q\|_{L^2} \geq \sigma$ for all $n$.  This
follows by resetting the initial time.)

We will induct on the statement ``$\egb(\lambda,\sigma)$ is false.''
Note that if $\lambda \leq \sigma_1\leq \sigma_2$, then
``$\egb(\lambda,\sigma_2)$ is false'' implies ``$\egb(\lambda,
\sigma_1)$ is false'', as is easily understood by writing down the
contrapositive.  We now define a threshold -- see the illustration
in Figure \ref{F:critical}.

\begin{figure}[h]
\includegraphics[width=400pt]{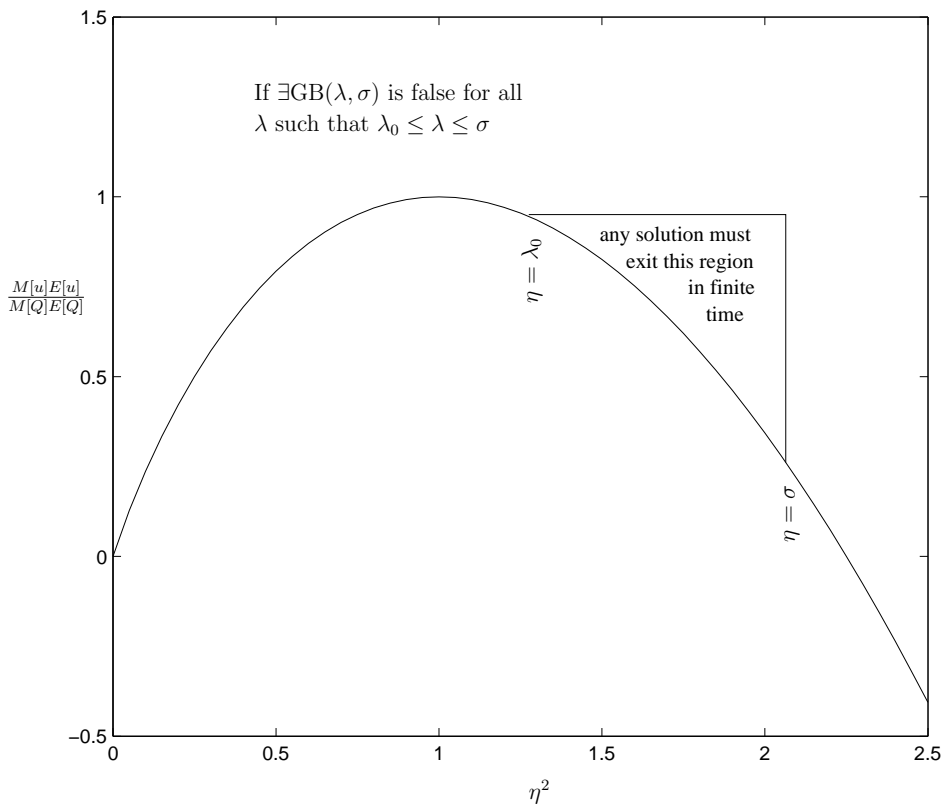}
\caption{A depiction of the meaning of the statement
``$\egb(\lambda,\sigma)$ is false for all $\lambda$ such that
$\lambda_0\leq \lambda\leq \sigma$.''  It means that for any
solution $u(t)$ with $\eta(t)>\lambda$ (when $\lambda$ is defined by
\eqref{E:energy-lambda}) if the path
$(\eta(t),3\lambda^2-2\lambda^3))$ is plotted here, it must escape
(along the horizontal line) the indicated triangular region at some
finite time.  The value $\sigma_c$ is the largest $\sigma$ for which
this statement holds. } \label{F:critical}
\end{figure}

\begin{definition}[The critical threshold]
Fix $\lambda_0>1$.  Let $\sigma_c=\sigma_c(\lambda_0)$ be the
supremum of all $\sigma>\lambda_0$ such that $\egb(\lambda,\sigma)$
\emph{is false} for all $\lambda$ such that $\lambda_0\leq \lambda
\leq \sigma$.  The notation $\sigma_c$ stands for
``$\sigma$-critical.''
\end{definition}

By Prop. \ref{P:nbc}, we know that $\sigma_c(\lambda_0) >\lambda_0$.

Suppose $\lambda_0>1$ and $\sigma_c(\lambda_0)=\infty$. Let $u(t)$
be any solution with $E[u]/E[Q]\leq 3\lambda_0^2-2\lambda_0^3$,
$M[u]=M[Q]$, and $\|\nabla u_0\|_{L^2}/ \|\nabla Q\|_{L^2}>1$.  We
claim there exists a sequence of times $t_n$ such that $\|\nabla
u(t_n)\|_{L^2}\to \infty$.  Indeed, suppose not, and let
$\lambda\geq \lambda_0$ be such that $E[u]/E[Q] =
3\lambda^2-2\lambda^3$.  Since there is no sequence $t_n$ along
which $\|\nabla u(t_n)\|_{L^2}\to +\infty$, there exists
$\sigma<\infty$ such that $\lambda \leq \|\nabla
u(t)\|_{L^2}/\|\nabla Q\|_{L^2} \leq \sigma$ for all $t\geq 0$.  But
this means that $\egb(\lambda, \sigma)$ holds true, and thus,
$\sigma_c(\lambda_0) \leq \sigma<\infty$.  Thus, in order to prove
our theorem, we need to show that for every $\lambda_0>1$, we have
$\sigma_c(\lambda_0)=\infty$.

Hence, we shall now fix $\lambda_0>1$ and assume that
$\sigma_c(\lambda_0)<\infty$, and work toward a contradiction.
Clearly, it suffices to do this for $\lambda_0$ close to $1$, and
thus, we shall make the assumption that $\lambda_0<\frac32$.  As
we'll see, this will be convenient later.

\section{Existence of a critical solution}

\begin{lemma}[Existence of a critical solution]
\label{L:exist-critical} There exist initial data $u_{c,0}$ and
$\lambda_c\in [\lambda_0, \sigma_c(\lambda_0)]$ such that
$u_c(t)\equiv \nls(t)\,u_{c,0}$ is global, $M[u_c]=M[Q]$, $E[u_c]=
3\lambda_c^2-2\lambda_c^3$ and
$$
\lambda_c \leq \frac{\|\nabla u_c(t)\|_{L^2}}{\|\nabla Q\|_{L^2}}
\leq \sigma_c \quad \text{for all }t\geq 0 \,.
$$
\end{lemma}

We have that for all $\sigma<\sigma_c$ and all $\lambda_0\leq
\lambda \leq \sigma$, $\egb(\lambda,\sigma)$ is false, i.e., there
are no solutions $u(t)$ for which $M[u]=M[Q]$,
$E[u]/E[Q]=3\lambda^2-2\lambda^3$ and
$$
\lambda \leq \frac{\|\nabla u(t)\|_{L^2}}{\|\nabla Q\|_{L^2}} \leq
\sigma \quad \text{for all }t\geq 0.
$$
But on the other hand, we have found a solution $u_c(t)$ such that
$M[u_c]=M[Q]$, $E[u_c]= 3\lambda_c^2-2\lambda_c^3$ and
$$
\lambda_c \leq \frac{\|\nabla u_c(t)\|_{L^2}}{\|\nabla Q\|_{L^2}}
\leq \sigma_c \quad \text{for all }t\geq 0 \,.
$$
Thus, we call this the ``critical solution'' or ``threshold
solution''. In \S\ref{S:concentration}, we shall show that these
properties induce a uniform-in-time concentration property of
$u_c(t)$, and  we then observe that all of the alleged properties of
$u_c(t)$ are inconsistent with the local virial identity (in
particular, Prop. \ref{P:local-blowup-time}).

\begin{proof}
By definition of $\sigma_c$, there exist sequences $\lambda_n$ and
$\sigma_n$ such that $\lambda_0\leq \lambda_n\leq \sigma_n$ and
$\sigma_n \searrow \sigma_c$ for which $\egb(\lambda_n, \sigma_n)$
holds.  This means that there exists $u_{n,0}$ with
$u_n(t)=\nls(t)\,u_{n,0}$ such that $u_n(t)$ is global,
$M[u_n]=M[Q]$, $E[u_n]/E[Q] =3\lambda_n^2-2\lambda_n^3$, and
$$
\lambda_n \leq \frac{\|\nabla u_n(t)\|_{L^2}}{\|\nabla Q\|_{L^2}}
\leq \sigma_n.
$$
Since $\lambda_n$ is bounded, we can pass to a subsequence such that
$\lambda_n$ converges.  Let $\lambda'=\lim_n \lambda_n$.  We know,
of course, that $\lambda_0\leq \lambda' \leq \sigma_c$.

In Lemma \ref{L:exist-nonscat}, take $\phi_n=u_{n,0}$, and
henceforth adopt the notation from that lemma.  For $M_1+1\leq j
\leq M_2$, the $v^j(t)$ scatter as $t\to +\infty$ and for $M_2+1\leq
j \leq M$, the $v^j$ also scatter in one or the other time direction
-- see Prop. \ref{P:pd}.  Thus, for all $M_1+1\leq j\leq M$, we have
$E[\psi^j]=E[v^j]\geq 0$.  By \eqref{E:energy-Pythag},
$$
\sum_{j=1}^{M_1} E[\psi^j] \leq E[\phi_n] + o_n(1).
$$
For at least one $1\leq j \leq M_1$, we have
$$
E[\psi^j] \leq \max\left( \lim_n  E[\phi_n], 0\right).
$$
We might as well take, WLOG, $j=1$.  Since we also have
$M[\psi^1]\leq \lim_n M[\phi_n]=M[Q]$, we have
$$
\frac{M[\psi^1]E[\psi^1]}{M[Q]E[Q]}\leq \max\left( \lim_n
\frac{E[\phi_n]}{E[Q]},0\right).
$$
Thus,
$$
\frac{M[\psi^1]E[\psi^1]}{M[Q]E[Q]} = 3\lambda_1^2-2\lambda_1^3.
$$
for some $\lambda_1 \geq \lambda_0$.
\footnote{This $\lambda_1$ is of course different from the $\lambda_1$ in
the sequence $\lambda_n$ used above.}  (In the case $\lim_n
E[\phi_n]\geq 0$, we will have $\lambda_1\geq \lambda'\geq
\lambda_0$.  In the case $\lim_n E[\phi_n]<0$, we will have
$\lambda_1>\frac32>\lambda_0$ but might not have $\lambda_1\geq
\lambda'$).  Since $v^1$ is a nonscattering solution, we cannot have
$\|\psi^1\|_{L^2}\|\nabla \psi^1\|_{L^2}<\|Q\|_{L^2}\|\nabla
Q\|_{L^2}$, since it would contradict Theorem  \ref{T:scattering}.
We therefore must have $\|\psi^1\|_{L^2}\|\nabla \psi^1\|_{L^2} >
\lambda_1 \|Q\|_{L^2}\|\nabla Q\|_{L^2}$.

Two cases emerge:

\medskip

\noindent\emph{Case 1}. $\lambda_1\leq \sigma_c$.  Since
$\egb(\lambda_1,\sigma_c-\delta)$ is false for each $\delta>0$ (the
inductive hypothesis), there exists a nondecreasing sequence $t_k$
of times such that
$$
\lim_k \frac{\|v^1(t_k)\|_{L^2}\|\nabla
v^1(t_k)\|_{L^2}}{\|Q\|_{L^2}\|\nabla Q\|_{L^2}} \geq \sigma_c.
$$
Hence,
\begin{equation}
\label{E:100}
\begin{aligned}
\sigma_c^2 - o_k(1)
& \leq \frac{\|v^1(t_k)\|_{L^2}^2\|\nabla v^1(t_k)\|_{L^2}^2}{\|Q\|_{L^2}^2\|\nabla Q\|_{L^2}^2} \\
&\leq \frac{\|\nabla v^1(t_k)\|_{L^2}^2}{\|\nabla Q\|_{L^2}^2} \\
&\leq \frac{\sum_{j=1}^M\|\nabla v^j(t_k-t_n^j)\|_{L^2}^2+\|\nabla W_n^M(t_k)\|_{L^2}^2}{\|\nabla Q\|_{L^2}^2} \quad \text{(recall that $t_n^1=0$)}\\
&\leq \frac{\|\nabla u_n(t)\|_{L^2}^2}{\|\nabla Q\|_{L^2}^2} + o_n(1)
\quad \text{(by Lemma \ref{L:Pythag-flow}, taking $n=n(k)$ large)}\\
&\leq \sigma_c^2+o_n(1).
\end{aligned}
\end{equation}
Send $k\to+\infty$ (and hence, $n(k)\to +\infty$).  We conclude that
all inequalities must be equalities.  In particular, we conclude
that $W_n^M(t_k)\to 0$ in $H^1$ norm \footnote{This implies that
$W_n^M(0)=W_n^M \to 0$ in $H^1$, since we know that $W_n^M(t)$ is a
scattering solution and have the bounds depicted in Fig.
\ref{F:dichotomy}. We do not need this observation for the current
proof, but do for the proof of Lemma \ref{L:compactness}.}, that
$v^j\equiv 0$ for all $j\geq 2$, and that $M[v_1]=M[Q]$.  Moreover,
by Lemma \ref{L:Pythag-flow}, we have that for all $t$,
$$
\frac{\|\nabla v^1(t)\|_{L_x^2}}{\|\nabla Q\|_{L^2}} \leq \lim_n
\frac{\|u_n(t)\|_{L_x^2}}{\|\nabla Q\|_{L^2}} \leq \sigma_c.
$$
Hence, we take $u_{c,0}=v^1(0)$ ($=\psi^1$), $\lambda_c=\lambda_1$.

\medskip

\noindent\emph{Case 2}.  $\lambda_1> \sigma_c$. Then we do not have
access to the inductive hypothesis, but we do know that for all $t$,
$$
\lambda_1^2 \leq \frac{\|v^1(t)\|_{L^2}^2\|\nabla
v^1(t)\|_{L^2}^2}{\|Q\|_{L^2}^2\|\nabla Q\|_{L^2}^2}.
$$
Replace the first line of \eqref{E:100} by the above inequality; the
rest of the inequalities in \eqref{E:100} still hold (we might as
well now take $t_k=0$).  Send $n\to +\infty$ to get $\lambda_1\leq
\sigma_c$, a contradiction.  Thus, this case does not arise.

\end{proof}

\section{Concentration of critical solutions}
\label{S:concentration}

In this section, we take  $u(t)=u_c(t)$ to be a critical solution,
as provided by Lemma \ref{L:exist-critical}.

\begin{lemma}
\label{L:compactness}
There exists a path $x(t)$ in $\mathbb{R}^3$ such that
$$K \equiv \{ \,u(t, \bullet -x(t)) \; | \; t\geq 0 \, \}\subset H^1$$
has compact closure in $H^1$.
\end{lemma}
\begin{proof}
As we showed in \cite[Appendix A]{DHR} it suffices to show that for
each sequence of times $t_n \to \infty$, there exists (passing to a
subsequence) a sequence $x_n$ such that $u(t_n, \bullet -x_n)$
converges in $H^1$.

Take $\phi_n=u(t_n)$ in Lemma \ref{L:exist-nonscat}. Arguing
similarly to the proof of Lemma \ref{L:exist-critical}, we obtain
that $\psi^j=0$ for $j\geq 2$ and $W_n^M\to 0$ in $H^1$ as $n \to
\infty$.  Hence, $u(t_n,\bullet-x_n)\to \psi^1$ in $H^1$.

\end{proof}

As a result of Lemma \ref{L:compactness}, we have a uniform-in-time
$H^1$ concentration of $u_c(t)$.

\begin{corollary}
\label{C:localization} For  each $\epsilon>0$, there exists $R>0$
such that for all $t$, $\|u(t,\bullet - x(t))\|_{H^1(|x|\geq R)}\leq
\epsilon$
\end{corollary}

The proof is elementary, but is given in \cite[Cor. 3.3]{DHR}.

We next observe that the localization property of $u_c(t)$ given by
Corollary \ref{C:localization} implies that $u_c(t)$ blows-up in
finite time by Prop. \ref{P:local-blowup-time}.  But this
contradicts the boundedness of $u_c(t)$ in $H^1$, and hence,
$u_c(t)$ cannot exist.

This contradiction completes the proof of Theorem \ref{T:main}.

\appendix

\section{Proof of Lemma \ref{L:center-control}}
\label{S:pf-center-control}

Here we will carry out the proof of Lemma \ref{L:center-control},
which closely follows the proof of Lemma 5.1 in \cite{DHR}.  We will
adopt the notation from that paper.

Without loss of generality we may assume that $x(0)=0$.  Let $\ds
R(T)= \max_{0\leq t\leq T} |x(t)|$.  It suffices to prove that there
is an absolute constant $c>0$ such that for each $T$ with
$R(T)=|x(T)|\gg 1$, we have
\begin{equation}
\label{E:x(T)-bound} |x(T)| \leq c\,T \left(e^{-|x(T)|} +
\epsilon\right)^2 \,.
\end{equation}
Consider such a $T>0$.  We know that for $0\leq t\leq T$, we have
$|x(t)|\leq R(T)$.  By (5.4) in \cite{DHR} (and adopting the
definition of $z_R(T)$ in that paper),  there is an absolute
constant $c_1$ such that
$$
|z_{2R(T)}'(t)| \leq c_1 \int_{|x|\geq 2R(T)} \left( |\nabla u(t)|^2
+ |u(t)|^2 \right) \,dx .
$$
By \eqref{E:apx-close},  for all $0\leq t\leq T$ there holds
$$
|z_{2R(T)}'(t)| \leq c_2 \left(\epsilon+\|Q\|_{H^1(|x|\geq R(T))}
\right)^2 .
$$
Owing to the exponential localization of $Q(x)$, we have upon
integrating the above inequality over $[0,T]$ the bound
\begin{equation}
 \label{E:step-a}
|z_{2R(T)}(t) - z_{2R(T)}(0)| \leq c_3 T \left(\epsilon+e^{-R(T)}
\right)^2 .
\end{equation}
Due to the fact that $|x(T)|=R(T)$, we have that there exists an
absolute constant $c_4$ such that
\begin{equation}
\label{E:step-b} |z_{2R(T)}(T)| \geq c_4 R(T) \,.
\end{equation}
Moreover, since $x(0)=0$, we have the simple bound
\begin{equation}
\label{E:step-c} |z_{2R(T)}(0)| \leq c_5 \left(1 + R(T)\epsilon^2
\right) .
\end{equation}
By combining \eqref{E:step-a}, \eqref{E:step-b}, and
\eqref{E:step-c}, we obtain \eqref{E:x(T)-bound}.

\section{Nonzero momentum}
\label{S:Galilean}

Suppose that we have a solution $u(x,t)$ with $M[u] = M[Q]$ and
$P[u]\neq 0$.  We apply the Galilean transformation to the solution
$u(t)$ as in Section 4 of \cite{DHR} to obtain a new solution
$\tilde u(x,t)$:
$$
\tilde u (x,t) = e^{i x \xi_0} e^{-it |\xi_0|^2} u(x-2\xi_0 t, t)
\quad \text{with} \quad \xi_0 = -\frac{P[u]}{M[u]}.
$$
Then
$$P[\tilde u] = 0\,, \quad M[\tilde u] = M[u] = M[Q] \,,$$
$$E[\tilde u] = E[u] - \frac12 \frac{P[u]^2}{M[u]} \,,$$
and
$$ \|\nabla \tilde u\|^2_2 = \|\nabla u \|^2_2 - \frac{P[u]^2}{M[u]} \,.$$
This choice of $\xi_0$ furnishes the lowest value of $E[\tilde u]$
under any choice of $\xi_0$.  It is easier to have $E[\tilde
u]<E[Q]$ than $E[u]<E[Q]$, suggesting that we should always
implement this transformation to maximize the applicability of Prop.
\ref{P:dichotomy}.  However, one should show for consistency that if
the dichotomy of Prop. \ref{P:dichotomy} was already valid for $u$
before the Galilean transformation was applied (i.e. $E[u]<E[Q]$),
then the selection of case (1) versus (2) in Prop. \ref{P:dichotomy}
is preserved.

Suppose $M[u]=M[Q]$, $E[u]<E[Q]$ and $P[u]\neq 0$.  Define $\tilde
u$ as above.  Let $\lambda_-$, $\lambda$ be defined in terms of
$E[u]$ by \eqref{E:energy-lambda}, and let $\eta(t)$ be defined in
terms of $u(t)$ by \eqref{E:eta-def}. Let $\tilde \lambda_-$,
$\tilde \lambda$ and $\tilde \eta(t)$ be the same quantities
associated to $\tilde u$.

Suppose  that case (1) of Prop. \ref{P:dichotomy} holds for $u$.
This implies, in particular, that $\eta(t) \leq 1$ for all $t$.  But
clearly $\tilde \eta(t) \leq \eta(t) \leq 1$, and thus, case (1) of
Prop. \ref{P:dichotomy} holds for $\tilde u$ also.

Now suppose that case (1) of Prop. \ref{P:dichotomy} holds for
$\tilde u$.  Then $\tilde \eta(t) \leq \tilde \lambda_-$ for all
$t$.  We must show that
$$
\eta^2 = \tilde \eta^2  + \frac{P[u]^2}{\|Q\|_{L^2}^2\|\nabla
Q\|_{L^2}^2} \leq \tilde \lambda_-^2 + \frac{P[u]^2}{6 M[Q]E[Q]}
\leq \lambda_-^2.
$$
This reduces to an algebraic problem.  For convenience, let $\alpha
= E[u]/E[Q]$ and $\beta = P[u]^2/M[Q]E[Q]$. Then $\tilde \lambda_-$
is the smaller of the two roots of $3\tilde \lambda_-^2 - 2\tilde
\lambda_-^3 = \alpha-\frac12\beta$ while $\lambda_-$ is the smaller
of the two roots of $3\lambda_-^2-2\lambda_-^3 = \alpha$.  In moving
$\rho$ forward from $\tilde \lambda_-^2$ to $\tilde \lambda_-^2 +
\frac16\beta$, we increment the function $3\rho-2\rho^{3/2}$ by an
amount at most $\frac16\beta (3-3\rho^{1/2})  \leq \frac12\beta$.
This completes the argument.

\end{document}